\documentclass[11pt,a4paper,reqno]{amsart}
\usepackage{amsthm,amsmath,amsfonts,amssymb,amsxtra,appendix,bookmark,dsfont,latexsym,bm,hyperref,delarray,color,euscript,amsgen,amsbsy,amsopn,amscd,latexsym,mathrsfs}

\makeatletter

\usepackage{hyperref}

\makeatletter
\def\@tocline#1#2#3#4#5#6#7{\relax
  \ifnum #1>\c@tocdepth 
  \else
    \par \addpenalty\@secpenalty\addvspace{#2}%
    \begingroup \hyphenpenalty\@M
    \@ifempty{#4}{%
      \@tempdima\csname r@tocindent\number#1\endcsname\relax
    }{%
      \@tempdima#4\relax
    }%
    \parindent\z@ \leftskip#3\relax \advance\leftskip\@tempdima\relax
    \rightskip\@pnumwidth plus4em \parfillskip-\@pnumwidth
    #5\leavevmode\hskip-\@tempdima
      \ifcase #1
       \or\or \hskip 1em \or \hskip 2em \else \hskip 3em \fi%
      #6\nobreak\relax
      \dotfill
      \hbox to\@pnumwidth{\@tocpagenum{#7}}
    \par
    \nobreak
    \endgroup
  \fi}
\makeatother

\newtheorem{theorem}{Theorem}[section]
\newtheorem{lemma}[theorem]{Lemma}

\newtheorem{corollary}[theorem]{Corollary}

\theoremstyle{definition}

\theoremstyle{remark}
\newtheorem{remark}[theorem]{Remark}

\newcommand\R{{\ensuremath {\mathbb R} }}

\newcommand\1{{\ensuremath {\mathds 1} }}

\newcommand\nn{\nonumber}

\renewcommand\phi{\varphi}

\newcommand{\bA}{\mathbf{A}}
\newcommand{\be}{\mathbf{e}}

\newcommand{\gH}{\mathfrak{H}}

\newcommand{\cS}{\mathcal{S}}

\newcommand{\cE}{\mathcal{E}}

\renewcommand{\epsilon}{\varepsilon}

\newcommand{\norm}[1]{ \left| \! \left| #1 \right| \! \right| }
\DeclareMathOperator{\tr}{{\rm Tr}}
\DeclareMathOperator{\Tr}{{\rm Tr}}

\renewcommand{\ge}{\geqslant}
\renewcommand{\le}{\leqslant}
\renewcommand{\geq}{\geqslant}
\renewcommand{\leq}{\leqslant}
\renewcommand{\hat}{\widehat}

\newcommand{\im}{\mathrm{i}}

\newcommand{\eps}{\varepsilon}

\numberwithin{equation}{section}

\begin{document}

\title{Improved stability for 2D attractive Bose gases}

\author[P.T. Nam]{Phan Th\`anh Nam}
\address{Department of Mathematics, LMU Munich, Theresienstrasse 39, 80333 Munich, and Munich Center for Quantum Science and Technology (MCQST), Schellingstr. 4, 80799 Munich, Germany} 
\email{nam@math.lmu.de}

\author[N. Rougerie]{Nicolas Rougerie}
\address{Universit\'e Grenoble-Alpes \& CNRS,  LPMMC (UMR 5493), B.P. 166, F-38042 Grenoble, France}
\email{nicolas.rougerie@grenoble.cnrs.fr}

\date{September, 2019}

\begin{abstract}
We study the ground-state energy of $N$ attractive bosons in the plane. The interaction is scaled for the gas to be dilute, so that the corresponding mean-field problem is a local non-linear Schr\"odinger (NLS) equation. We improve the conditions under which one can prove that the many-body problem is stable (of the second kind). This implies, using previous results, that the many-body ground states and dynamics converge to the NLS ones for an extended range of diluteness parameters.   
\end{abstract}

\maketitle


\section{Introduction}

Consider a 2D non-relativistic gas of bosonic particles trapped in the external potential 
$$ V:\R^2 \mapsto \R^+, \quad V(x) \underset{|x|\to \infty}{\longrightarrow} + \infty$$
and plunged in an external magnetic field $B$ of vector potential 
$$\bA:\R^2 \mapsto \R^2$$
such that $\mathrm{curl} \, \bA = B$. The particles interact via the pair potential of the form 
\begin{equation}\label{eq:wN}
\frac{1}{N-1} N^{2\beta} w (N^{\beta} x), 
\end{equation}
with a fixed ($N$-independent) parameter $\beta >0$ and a fixed function
$$ w:\R^2 \mapsto \R, \quad w(x) = w(-x), \quad w(x) \underset{|x|\to\infty}{\longrightarrow} 0.$$
Mathematically this means looking at the action of the many-body Schr\"odinger operator
\begin{equation}\label{eq:hamil}
H_N := \sum_{j=1} ^N \left\lbrace \left( -\im \nabla_j + \bA (x_j)\right)^2 + V(x_j) \right\rbrace+ \frac{1}{N-1} \sum_{1\leq i < j \leq N} N^{2\beta} w (N^{\beta} (x_i-x_j))
\end{equation}
on the symmetric space 
$$\gH_N=L^2_{\rm sym} (\R^{2N})=\bigotimes^N_{\rm sym} L^2(\R^2).$$
For nice data $V,w,\bA$, this operator is bounded from below with the core domain $\gH_N \cap C_c^\infty(\R^{2N})$, and thus can be extended to be a self-adjoint operator by Friedrichs' method. The non-trivial question is then that of stability of the second kind~\cite{LieSei-09}, i.e. whether $H_N \ge -CN$ for a constant $C$ independent of $N$. Note that when the negative part $w_-:=\min\{w,0\}$ is nonzero, the interaction energy with the rescaled potential in~\eqref{eq:wN} may be very negative, and we will see that the stability question cannot be  answered using merely the uncertainty principle. 

Out of $H_N$ one constructs a corresponding non-linear Schr\"odinger (NLS) functional 
\begin{equation}\label{eq:nls}
\cE ^{\rm nls} [u] = \lim_{N\to \infty} \frac{1}{N} \langle u^{\otimes N} |H_N| u^{\otimes N} \rangle = \langle u | h | u \rangle_{L^2} + \frac{b}{2} \int_{\R^2} |u|^4  
\end{equation}
where 
$$ b= \int_{\R^2} w$$
and 
\begin{equation}\label{eq:hamil onebody} 
h= \left( -\im \nabla + \bA (x)\right)^2 + V 
\end{equation}
is the one-body Hamiltonian. Clearly, $\cE^{\rm nls}$ is bounded from below under the constraint $\|u\|_{L^2(\R^2)}=1$ if and only if 
\begin{equation}\label{eq:stab nls}
 b \geq - a^*
\end{equation}
where $a^*$ is the optimal constant for the Gagiardo-Nirenberg inequality 
\begin{align} \label{eq:GN-inequality} \left( \int_{\R^2} |u|^2 \right) \left( \int_{\R^2} |\nabla u|^2 \right) \ge \frac{a^*}{2}\int_{\R^2} |u|^4, \quad \forall u\in H^1(\R^2).
\end{align}
See~\cite{Weinstein-83,GuoSei-13,Maeda-10,Kwong-89,Frank-14} for references.

Given the variational construction \eqref{eq:nls}, Condition~\eqref{eq:stab nls} is necessary for $H_N$ to be stable of the second kind. Sufficient conditions are in fact more stringent. Simple considerations show that we must demand more than~\eqref{eq:stab nls}. Here we work with the condition 
\begin{equation}\label{eq:true stab}
 \int_{\R^2} |w_-| < a^*
\end{equation}
and refer to~\cite{LewNamRou-14c} for a refinement (Hartree stability). 

In this note we are interested in the range of diluteness parameter $\beta >0$ for which the stability of the second kind can be shown to hold. Indeed, $\beta$ measures how fast the interactions converge to point-like ones. Implicit in the above is that the reference length-scale of the system is fixed, set by that of the one-body Hamiltonian $h$ (think of particles in a fixed box if you wish). Hence $N^{-\beta}$ measures the range of the interaction potential in units of the reference length scale. 
From the point of view of interactions, $N^{\beta -1/2}$ is the average number of particles a tagged one interacts with at a time, for $N^{-1/2}$ is the mean inter-particle distance. 


It is well-known that large quantum interacting systems are harder to deal with for large values of $\beta$, in particular, in our 2D case, for $\beta > 1/2$ which is the threshold to have a dilute system (few but strong inter-particle collisions). In this note we make the remark that a combination of the tools in~\cite{LewNamRou-15,Rougerie-19} allows to prove stability of the second kind under the condition $ \beta < 1.$ This extends the range of validity of methods~\cite{LewNamRou-14c,LewNamRou-15,CheHol-15,JebPic-17} dealing with the large-$N$ limit of $H_N$.

\bigskip

\noindent\textbf{Acknowledgements.} We thank Mathieu Lewin and Fernando G.S.L. Brand\~ao for helpful discussions. We received fundings from the European Research Council (ERC) under the European Union's Horizon 2020 Research and Innovation Programme (Grant agreement CORFRONMAT No 758620), and from the
Deutsche Forschungsgemeinschaft (DFG, German Research Foundation) under Germany's Excellence Strategy (EXC-2111-390814868).

\section{Main result}

For simplicity we assume that $V,\bA,w$ are smooth. The actual condition on the interaction potential we need is 
\begin{equation}\label{eq:asum w}
w\in L^1(\R^2)\cap L^2(\R^2).
\end{equation}
We also assume that there exist positive constants  $s>0$ and $c >0$ such that 
\begin{equation}\label{eq:trapping}
V(x) \geq c^{-1} |x|^s - c   
\end{equation}
and
\begin{equation}\label{eq:mag}
|\bA(x)| \le c e^{c|x|}.   
\end{equation}

Our main result is 

\begin{theorem}[\textbf{Stability of attractive 2D Bose gases}]\label{thm:main}\mbox{}\\
Under the above assumptions \eqref{eq:true stab}-\eqref{eq:asum w}-\eqref{eq:trapping}-\eqref{eq:mag}, for any fixed $0<\beta < 1$  there exists a constant $C>0$ such that  
\begin{equation}\label{eq:stability}
H_N \geq - C N 
\end{equation}
as an operator acting on $\gH_N = L^2_{\rm sym} (\R^{2N})$.
\end{theorem}

\begin{proof}[Remarks]\mbox{}\\
\noindent\textbf{1.} The best result preceding the above is that from~\cite{LewNamRou-15} which covers 
$$ \beta < \frac{s+1}{s+2}$$
with $s$ the exponent in~\eqref{eq:trapping}. This had the merit of allowing a dilute gas, $\beta >1/2$, whereas the previous results~\cite{LewNamRou-14c} were limited to $\beta < 1/2$. It is not our aim to pretend that increasing $\beta >0$ is an undertaking that should go on forever, but there are reasons that make us feel the above is noteworthy. First, the annoying dependence on the trapping potential gets dispensed with (provided the growth is still polynomial). Second, the proof is somewhat cleaner. Third, a natural barrier seems to have been reached: the condition $\beta <1$ is that needed to obtain a second-moment estimate following the techniques of~\cite{ErdYau-01,NamRouSei-15,LewNamRou-15}. 

\medskip

\noindent\textbf{2.} It is not obvious to us whether there should exist a physically natural upper bound on $\beta$ in the attractive case. For 3D repulsive gases this would be given by the Gross-Pitaevskii~\cite{LieSeiSolYng-05,Rougerie-EMS} limit $\beta = 1$. For 2D repulsive gases, any $\beta >0$ should be allowed, for the GP limit~\cite{LieSeiYng-01,JebLeoPic-16} corresponds to an interaction scaled exponentially with $N$.

\medskip

\noindent\textbf{3.} The proof proceeds by combining the second moment estimate of~\cite{LewNamRou-15} with  the information-theoretic quantum de Finetti theorem~\cite{BraHar-12,LiSmi-15} of Brand\~ao-Harrow. The interest of the latter for large bosonic systems~\cite{Rougerie-19} is that it (almost) gives a quantitative de Finetti theorem when the one-body state-space is infinite dimensional (more precisely the dimension should be still finite but only its logarithm enters relevant estimates), thus bypassing the main technical limitation of the tools used in~\cite{LewNamRou-15}.

%

\medskip

\noindent\textbf{4.} For the above stability result we may assume $\bA=0$, thanks to the diamagnetic inequality~\cite[Theorem 7.21]{LieLos-01}, as well as ignore the bosonic symmetry (i.e. the lower bound holds true on the full space $L^2(\R^{2N})$). However, the presence of the magnetic field and the Bose-Einstein statistics are meaningful for the Corollary \ref{cor:GS} below. 
\end{proof}

The stability result has two corollaries regarding the large $N$ limit of the system at hand. First for ground states:

\begin{corollary}[\textbf{NLS limit for ground states}]\label{cor:GS}\mbox{}\\
Under the above assumptions, the ground state energy per particle 
$$\frac{E(N)}{N} = \frac{1}{N}\inf \sigma_{\gH_N} (H_N)=\frac{1}{N} \inf \{ \langle \Psi, H_N \Psi\rangle : \|\Psi\|_{\gH_N}=1 \}$$
converges when $N\to \infty$ to the ground state energy of the NLS functional in \eqref{eq:nls}, i.e. 
$$
E^{\rm nls}= \inf\{ \cE^{\rm nls}(u): \|u\|_{L^2(\R^2)}=1\}.
$$
Moreover, the reduced density matrices $\{\gamma_{\Psi _{N}}^{(k)}\}_{N}$ of ground states $\Psi_N$ of $H_N$ converge to convex combinations
of projections on NLS minimizers, namely there exists a Borel probability measure $\mu$ supported on the minimizers of $E^{\rm nls}$ such that, along a subsequence $N\to \infty$, 
\begin{align} \label{eq:thm-cv-DM}
\lim_{N \to \infty}\Tr \left| \gamma_{\Psi _{N}}^{(k)} - \int |u^{\otimes k} \rangle \langle u^{\otimes k}| d\mu(u) \right| =0,\quad \forall k\in \mathbb{N}.
\end{align}
\end{corollary}

\medskip

Next, for dynamics:

\begin{corollary}[\textbf{NLS limit for dynamics}]\label{cor:dyn}\mbox{}\\
Under the above assumptions and with $\bA \equiv 0$, the many-body Schr\"odinger dynamics 
$$ i\partial _t \Psi_N(t) = H_N \Psi_N(t)$$
starting from a well-prepared initial datum, e.g. $\Psi_N(0)=u(0)^{\otimes N}$ with $u(0)$ smooth,  converges when $N\to \infty$ to the NLS dynamics
$$ i \partial_t u(t) = h u(t) + b |u(t)|^2 u(t)$$
in the sense of reduced density matrices
\begin{align} \label{eq:thm-cv-DM-dynamics}
\lim_{N \to \infty}\Tr \left| \gamma_{\Psi _{N}(t)}^{(k)} - |u(t)^{\otimes k} \rangle \langle u(t)^{\otimes k}| \right| =0,\quad \forall k\in \mathbb{N}, \quad \forall t\in \R. 
\end{align}
\end{corollary}

\begin{proof}[Remarks]\mbox{}\\
\noindent\textbf{1.} Recall that  the $k$-particles reduced density matrix of any $\Psi \in\gH_N$  is defined as
$$\gamma_{\Psi}^{(k)}:= \Tr_{k+1\to N} |\Psi \rangle \langle \Psi|.$$
Here $\tr_{k+1\to N}$ means the partial trace over $N-k$ factors of $\gH_N = \bigotimes_{\rm sym} ^N L^2(\R^2)$. In particular, for the product state we have $\gamma_{u^{\otimes N}}^{(k)}=|u^{\otimes k}\rangle \langle u^{\otimes k}|$. Thus the convergences \eqref{eq:thm-cv-DM}, \eqref{eq:thm-cv-DM-dynamics} tells us that the corresponding many-body wave functions are close to product states in a weak sense (it is well known that they are not close in $L^2$-norm). 

\medskip

\noindent\textbf{2.} Although the statements cover also repulsive gases ($w\geq 0$), the novelty lies mainly in the attractive case. If $w\geq 0$ the proof is much simpler indeed and one can reach much higher values of $\beta$; see ~\cite{LieSeiYng-01} for ground states and~\cite{JebLeoPic-16} for dynamics.  

\medskip

\noindent\textbf{3.} We do not even sketch the proof of Corollary~\ref{cor:GS}. The reader should have no difficulty in figuring out that, given Theorem~\ref{thm:main}, the proof of~\cite[Theorem~1]{LewNamRou-15} applies mutatis mutandis for $\beta < 1$, at least under the additional condition $|\bA(x)|^2 \le V(x)$. The latter technical condition can be replaced by~\eqref{eq:mag} following arguments from \cite[Section 4.2]{NamRouSei-15}. 

\medskip

\noindent\textbf{4.} For the dynamical statement, the method of~\cite{JebPic-17} allows to derive NLS dynamics provided~\eqref{eq:stability} holds. Hence, Theorem~\ref{thm:main} extends the range of validity of their main result to all $\beta < 1$. Such a range was obtained previously in~\cite{NamNap-17} (without using~\eqref{eq:stability}), and a more restricted one in~\cite{CheHol-15} (in the special case $s=2$).  We also refer to these references for precise descriptions of the well-prepared initial data that can be covered. 

%
%
\end{proof}

\begin{remark}[Classically stable case in 3D]\mbox{}\\
Stability of the second kind is also an issue for 3D Bose gases with potentials 
$$ N^{3\beta - 1} w(N^{\beta} x)$$
having an attractive part, in the dilute regime $\beta > 1/3$. Because of the respective scalings of the mean-field interaction and kinetic energy one must then assume~\cite{LewNamRou-14c,Triay-17} that the unscaled potential $w$ is classically stable 
$$ \sum_{1\leq i < j \leq N} w (x_i - x_j) \geq - C N$$
for all $x_1,\ldots,x_N \in \R ^d$. This holds for exampe if $\hat{w} \geq 0$. 

Our method also covers this case, but we refrain from stating details, for the range of $\beta$ we are allowed to reach is $1/3 < \beta < 9/26$. Triay, in his study of the dipolar Bose gas~\cite{Triay-17}, has already obtained $1/3 < \beta < 1/3 + s /(45+42s)$ where $s$ is the exponent in~\eqref{eq:trapping}. Only for rather small values of $s$ does our estimate improve on his. The 3D case for a restricted class of potentials with an attractive part has also been considered~\cite{Lee-09,Yin-10b} for homogeneous gases in the thermodynamic limit. \hfill$\diamond$
\end{remark}

\section{Proof} 

This being for a large part an improvement on~\cite{LewNamRou-15,Rougerie-19}, we shall be brief. See~\cite{BenPorSch-15,Golse-13,LieSeiSolYng-05,Rougerie-EMS,Rougerie-LMU,Rougerie-spartacus,Schlein-08} for general background on large $N$ limits of bosonic quantum systems.

Recall that we write the proof for $\bA \equiv 0$, which implies the general statement using the diamagnetic inequality~\cite[Theorem 7.21]{LieLos-01}.

%

\subsection{Old arguments} Consider the spectral projectors
\begin{equation}\label{eq:spec}
 P = \1_{h\leq \Lambda}, \quad Q = \1 - P
\end{equation}
with $\Lambda > 0$ a (large enough) one-body energy cut-off. Let 
$$ H_2 := h_1 + h_2 + N^{2\beta} w(N^{\beta} (x_1-x_2)) $$
be the two-body Hamiltonian associated with~\eqref{eq:hamil}. Let $\Psi_N$ be a ground state of $H_N$ and let $\gamma_{N}^{(k)}$ be its $k$-body density matrix. Then
\begin{align} \label{eq:EN-gammaN} \frac{1}{N}E(N) = \frac{1}{N} \langle \Psi_N |H_N| \Psi_N\rangle = \frac{1}{2}\tr \left( H_2 \gamma_N ^{(2)}\right).
\end{align}

We start with a localization lemma, essentially a restatement of~\cite[Equation (46)]{LewNamRou-15}:

\begin{lemma}[\textbf{Localization}]\label{lem:loc}\mbox{}\\
With the above notation, for any $\delta > 1/2$ there exists a $C_\delta >0$ such that  
$$
\tr \left( \left( H_2 - P^{\otimes 2} H_2 P^{\otimes 2} \right) \gamma_{\Psi_N} ^{(2)} \right) \geq - C_{\delta} \Lambda ^{(\delta -1)/2} \left( \tr \left(h \,\gamma_{\Psi_N} ^{(1)}\right) \right) ^{(1-\delta)/2}  \left( \tr \left(h\otimes h \,\gamma_{\Psi_N} ^{(2)}\right) \right) ^{\delta}. 
$$
\end{lemma}

To put the above to good use we need a priori bounds on the first and second moments of the one-body Hamiltonian. The following is~\cite[Lemma~5]{LewNamRou-15}, and this is where we use our main assumption $\beta < 1$ and also the requirement $w\in L^2 (\R^2)$. 

\begin{lemma}[\textbf{Moments}]\label{lem:mom}\mbox{}\\
Let $0<\beta <1$. For all $\eps\in (0,1)$ we have 
\begin{equation}\label{eq:one-two body bound}
\Tr\left(h \,\gamma_{\Psi_N}^{(1)}\right) \leq C\frac{1 + |e_{N,\eps}|}{\eps}\quad \text{and} \quad  \Tr \left(h\otimes h\, \gamma_{\Psi_N}^{(2)}\right) \leq C \left( \frac{1 + |e_{N,\eps}|}{\eps} \right) ^2
\end{equation}
where 
\begin{equation}\label{eq:ener epsilon}
e_{N,\eps}:= N^{-1} \inf_{\Psi \in \gH_N, \|\Psi\|=1}  \left\langle \Psi \left|  H_N - \eps \sum_{j=1}^N h_j  \right| \Psi \right\rangle. 
\end{equation}
\end{lemma}

\subsection{New argument}

We rely on a version of the quantum de Finetti theorem from~\cite{BraHar-12,LiSmi-15}:

\begin{lemma}[\textbf{de Finetti}]\label{lem:deF}\mbox{}\\
Let $\gH$ be a complex separable Hilbert space, and $\gH_N = \gH ^{\otimes_{\rm sym} N}$ the corresponding bosonic space. Let $\gamma_N ^{(2)}$ be the $2$-body reduced density matrix of a $N$-body state vector $\Psi_N \in \gH_N$ (or a general mixed state). 

Let $P$ be a finite dimensional orthogonal projector. There exists a Borel measure $\mu_N ^{(2)}$ with total mass $\leq 1$ on the set of one-body mixed states
\begin{equation}\label{eq:one body states}
\cS_P := \left\{ \gamma \mbox{ positive trace-class operator on } P \gH, \, \tr \gamma = 1 \right\} 
\end{equation}
such that 
\begin{equation}\label{eq:deF rep}
\tr\left| A \otimes B \left(  P ^{\otimes 2} \gamma_N ^{(2)} P ^{\otimes 2} - \int \gamma ^{\otimes 2} d\mu_N ^{(2)} (\gamma) \right)\right| \leq C \sqrt{\frac{\log (\dim (P))}{N}} \norm{A}\norm{B}
\end{equation}
for all $A,B$ self-adjoint operators on $P\gH$. The norm in the right-hand side is the operator norm.
\end{lemma}

\begin{proof}
The proof of~\cite{BraHar-12} gives the statement with $A,B$ replaced by quantum measurements. In~\cite[Proposition~3.2]{Rougerie-19} it is explained how the statement with $A,B$ positive operators follows. The full result is obtained by decomposing self-adjoint operators in the manner
$$ A = \1_{A < 0} A + \1_{A \geq 0} A$$
as used already in~\cite{Girardot-19}.
\end{proof}

To apply the above we shall, as in~\cite{Rougerie-19}, decompose the interaction operator using the Fourier transform in the manner
\begin{align}\label{eq:decomp w}
 N^{2\beta} w (N^\beta(x-y)) &= \int_{\R^2}  N^{2\beta} \widehat{ w (N^\beta\cdot )} (k) e^{i k\cdot x} e^{-i k\cdot y}dk \nonumber\\
 &= \int_{\R^2}  \widehat{w}(N^{-\beta}k) \left( \cos \left( k\cdot x \right) \cos \left( k \cdot y \right) + \sin \left( k\cdot x \right) \sin \left( k \cdot y \right) \right) dk
\end{align}
and apply Lemma~\ref{lem:deF} for each $k$. We interject a simple control of the involved multiplication operators:

\begin{lemma}[\textbf{Multiplication by plane waves}]\label{lem:cos}\mbox{}\\
Let $k\in \R^2, k\neq 0$ and $\be_k$ be the multiplication operator on $L^2 (\R^2)$ by either $\cos (k \cdot x)$ or $\sin (k \cdot x)$. Let $P$ be the spectral projector in~\eqref{eq:spec}. As operators
\begin{equation}\label{eq:cos}
 \pm P \be_k P \leq \min\left\{ 1, C \frac{\Lambda^{1/2}}{|k|} \right\}. 
\end{equation}
\end{lemma}

\begin{proof} The first upper bound $1$ is obvious as $|\be_k|\le 1$. We write the proof of the second bound for $\be_k$ the multiplication by $\cos(k\cdot x)$ (the case of $\sin(k\cdot x)$ is similar). For any smooth compactly supported function $f$, integrating by parts, 
$$ \langle f | \be_k  | f \rangle = - \int_{\R^2} \frac{\sin (k\cdot x)}{|k|} \frac{k}{|k|} \cdot \nabla \left(|f| ^2\right),$$
whence 
\begin{align*}
 \left| \langle f | \be_k  | f \rangle \right| \leq \frac{2}{|k|} \int_{\R^2} |f| |\nabla f| \le \frac{2}{|k|} \left( \int_{\R^2} |f|^2 \right)^{1/2} \left( \int_{\R^2} |\nabla f|^2 \right)^{1/2}. 
\end{align*}
Applying the above with $f = P g$, $g\in L^2 (\R^d)$ proves the claim (recall we work without a magnetic field). 
\end{proof}

We can finally give the 

\begin{proof}[Proof of Theorem~\ref{thm:main}] 
According to Lemma~\ref{lem:loc} and Lemma~\ref{lem:mom}   we have  
\begin{align}\label{eq:pre pouet}
 \tr \left( H_2 \gamma_N ^{(2)}\right) &\geq \tr \left( P^{\otimes 2} H_2 P^{\otimes 2}\gamma_N ^{(2)}\right) \nn\\ &\qquad - C_{\delta} \Lambda ^{(\delta -1)/2} \left( \tr \left(h \gamma_N ^{(1)}\right) \right) ^{(1-\delta)/2}  \left( \tr \left(h\otimes h \gamma_N ^{(2)}\right) \right) ^{\delta} \nn\\
 &\ge \tr \left( P^{\otimes 2} H_2 P^{\otimes 2}\gamma_N ^{(2)}\right) - C_{\delta} \Lambda ^{(\delta -1)/2} \left( \frac{1 + |e_{N,\eps}|}{\eps}\right) ^{(1+3\delta) /2}
\end{align}
for any $1/2<\delta \le 1$ and $0<\eps<1$. On the other hand we may use Lemma~\ref{lem:deF}: denoting 
$$ \widetilde{\gamma_N} ^{(2)} := \int \gamma^{\otimes 2} d\mu_N ^{(2)} (\gamma)$$
we have 
\begin{multline}\label{eq:pouet}
 \tr \left( P^{\otimes 2} H_2 P^{\otimes 2}\gamma_N ^{(2)}\right) \geq \tr \left( P^{\otimes 2} H_2 P^{\otimes 2}\widetilde{\gamma_N} ^{(2)}\right) \\- C \sqrt{\frac{\log (\Lambda)}{N}} \left( \Lambda + \sum_{\be_k\in \{\cos(k\cdot x), \sin (k\cdot x)\}} \int_{\R^2} \norm{P \be_k P} ^2 |\widehat{w}(N^{-\beta} k)| dk \right).
\end{multline}
Here we have decomposed the interaction term as in~\eqref{eq:decomp w}, used the triangle inequality with~\eqref{eq:deF rep} and recalled that $\dim (P)$ depends at worst polynomially on $\Lambda$ as recalled e.g. in~\cite[Lemma~3.3]{LewNamRou-14c} (this is the place where Condition \eqref{eq:trapping} is used). The error term also contains the operator norm $2\Lambda$ of $P^{\otimes 2} (h_1 + h_2) P^{\otimes 2}$.

The second term on the right side of \eqref{eq:pouet} can be estimated using Lemma~\ref{lem:cos} and Condition ~\eqref{eq:asum w} (the latter ensures that $\hat w\in L^2\cap L^\infty$):
\begin{align*}
 & \int_{\R^2} \norm{P \be_k P} ^2 |\widehat{w}(N^{-\beta} k)| dk  \le  \int_{\R^2} \min\{1, C\Lambda|k|^{-2}\} |\widehat{w}(N^{-\beta} k)| dk \\
 & \le \int_{|k|\le 1} \|\hat w\|_{L^\infty} dk + C\int_{1<|k|\le N^{\beta}} \Lambda |k|^{-2} \|\hat w\|_{L^\infty} dk + C\int_{|k|>N^{\beta}}  \Lambda |k|^{-2} |\widehat{w}(N^{-\beta} k)| dk\\
 & \le C +  C \Lambda \log N + C \Lambda. 
\end{align*}

As regards the first term on the right-hand side of \eqref{eq:pouet} we write 
$$ \tr \left( P^{\otimes 2} H_2 P^{\otimes 2}\widetilde{\gamma_N} ^{(2)}\right) = \int \cE^{\rm H}[\gamma] d\mu_N^{(2)}(\gamma)$$
where 
$$ \cE^{\rm H}[\gamma] =  \tr \left(h\gamma\right) +\frac{1}{2} \iint_{\R^2 \times \R^2} \rho_{\gamma} (x) N^{2\beta} w(N^{\beta}( (x-y)) \rho_{\gamma} (y) dxdy.$$
Here the density $\rho_\gamma(x)=\gamma(x,x)$ (defined properly by the spectral decomposition) satisfies 
$$\int_{\R^2}\rho_\gamma  = \Tr \gamma =1.$$  
Using the diamagnetic inequality~\cite[Theorem~7.21]{LieLos-01} and the convexity of the kinetic energy~\cite[Theorem~7.8]{LieLos-01} we have  the Hoffmann-Ostenhof-type  inequality  
$$ \tr\left(h\gamma\right) \geq \int_{\R^2} |\nabla \sqrt{\rho_\gamma}|^2.$$
Then inserting the Cauchy-Schwarz inequality in the interaction term, 
$$
N^{2\beta} w(N^{\beta}( (x-y)) \rho_\gamma(x) \rho_{\gamma} (y) dxdy \ge - N^{2\beta} |w_-(N^{\beta}( (x-y))| \frac{\rho_\gamma(x)^2+\rho_\gamma(y)^2}{2},
$$
and combining with the Gagiardo-Nirenberg inequality \eqref{eq:GN-inequality} and Assumption~\eqref{eq:true stab}, we get
\begin{equation} \label{eq:eH>0}
\cE^{\rm H}[\gamma] \ge \int_{\R^2} |\nabla \sqrt{\rho_\gamma}|^2 - \frac{1}{2} \left(\int |w_-| \right) \int_{\R^2} \rho_\gamma^2  \ge 0. 
\end{equation}
Thus the first term on the right side of \eqref{eq:pouet} is nonnegative. Therefore,  \eqref{eq:pre pouet} reduces to 
\begin{align*} 
e_N =  \tr \left( H_2 \gamma_N ^{(2)}\right)  \ge -   C \sqrt{\frac{\log (\Lambda)}{N}} \Lambda \log N   -  C_{\delta} \Lambda ^{(\delta -1)/2} \left( \frac{1 + |e_{N,\eps}|}{\eps}\right) ^{(1+3\delta) /2}.
\end{align*}
Moreover by a simple trial state argument we know that $e_N\le C$. In summary, we have
\begin{align}\label{eq:post pouet}
|e_N| \le C + C \sqrt{\frac{\log (\Lambda)}{N}} \Lambda \log N   +  C_{\delta} \Lambda ^{(\delta -1)/2} \left( \frac{1 + |e_{N,\eps}|}{\eps}\right) ^{(1+3\delta) /2}
\end{align}
for all $1/2<\delta \le 1$ and $0<\eps<1$. 

Now we bootstrap the above arguments. Let $\eps_0\in (0,1)$ be a fixed constant such that 
$$
\int_{\R^2} {w_-} > -a^* (1-\eps_0). 
$$
Assume we know that there is a $\alpha >0$ such that 
\begin{equation}\label{eq:induction}
 |e_{N,\eps}|\leq C_{\eps,\alpha} N^{\alpha} \mbox{ for all } 0< \eps < \eps_0.
\end{equation}
We can start the bootstrap from $\alpha = 2\beta$, using the simple one-body inequality\footnote{Here we do not need the improved bound $O(N^{2\beta-1})$ in ~\cite[Lemma~2]{LewNamRou-15} which requires $\hat w\in L^1$.} 
$$-\Delta_x - N^{2\beta} w(N^\beta x) \ge - C N^{2\beta}$$
Then, we can apply the above arguments to the $\eps$-perturbed Hamiltonian in~\eqref{eq:ener epsilon} to bound $e_{N,\eps}$. Note that the lower bound \eqref{eq:eH>0} remains valid with $w$ replaced by $(1-\eps)^{-1}w$, provided that $1<\eps<\eps_0$. Combining with \eqref{eq:induction} we obtain 
$$
|e_{N,\eps}| \le C_\eps + C_{\eps} \sqrt{\frac{\log (\Lambda)}{N}} \Lambda \log N   +  C_{\delta,\eps} \Lambda ^{(\delta -1)/2} N^{\alpha(1+3\delta) /2}, \quad \forall 0<\eps<\eps_0. 
$$
If we choose $\Lambda = N^{\alpha + a}$ and $\delta = 1/2 + b \in (1/2,1)$ in the above we obtain 
$$ |e_{N,\eps}| \leq C_{\eps,a,b} \left( 1 + N^{\alpha +a - 1/2} \log N  + N ^{\alpha + 2b\alpha + ab/2 - a/4} \right).$$
Clearly we can pick $0<a < 1/2$ and $b>0$ small enough so that this implies
$$ |e_{N,\eps}| \leq C_{\eps} \left( 1 + N ^{\alpha - c}\right)$$
for a fixed constant $c>0$ (independent of $N,\eps,\alpha$). Thus we have shown that if~\eqref{eq:induction} holds with $\alpha>0$, it also holds with $\alpha$ replaced by $\max(\alpha -c,0).$ After finitely many steps of this procedure, we deduce that actually $|e_{N,\eps}|$ is bounded independently of $N$, which implies the theorem.
\end{proof}


\end{document}